\documentclass[12pt]{article}

\usepackage{diagbox}
\usepackage{mathtools}
\usepackage{bbm}
\usepackage{latexsym}
\usepackage{epsfig}
\usepackage{amsmath,amsthm,amssymb,enumerate,hyperref}
\usepackage[active]{srcltx}
\usepackage{color}

\newcounter{rot}

\newcommand{\hatt}[1]{\widehat #1}

\def\hc{\hatt{c}}

\def\a{\alpha} \def\b{\beta}  \def\D{\Delta}
\def\e{\varepsilon}    \def\g{\gamma}
\def\G{\Gamma}  
     \def\l{\lambda}

\def\cP{{\cal P}}

\newtheorem{theorem}{Theorem}
\newtheorem{lemma}[theorem]{Lemma}
\newtheorem{corollary}[theorem]{Corollary}
\newtheorem{Remark}{Remark}

\newcommand{\brac}[1]{\left(#1\right)}

\newcommand{\bfrac}[2]{\left(\frac{#1}{#2}\right)}

\newcommand{\set}[1]{\left\{#1\right\}}

\newcommand{\ind}[1]{\mbox{{\large 1}} \brac{#1}}

\def\E{\mathbb{E}}

\def\Pr{\mathbb{P}}

\newcommand{\ignore}[1]{}

\newcommand{\beq}[2]{\begin{equation}\label{#1}#2\end{equation}}
\newcommand{\mults}[1]{\begin{multline*}#1\end{multline*}}

\def\dd{\text{d}}

\usepackage{tikz}
\usetikzlibrary{decorations.pathmorphing}
\usetikzlibrary{positioning}
\usetikzlibrary{arrows,automata}
\usetikzlibrary{shapes.misc}
\usetikzlibrary{backgrounds}
\usetikzlibrary{arrows,shapes}

\begin{document}
\author{Alan Frieze\thanks{Research supported in part by NSF grant DMS-1952285} \  and Tomasz Tkocz\thanks{Research supported in part by NSF grant DMS-1955175} \\Department of Mathematical Sciences\\Carnegie Mellon University\\Pittsburgh PA 15213}

\title{Shortest paths with a cost constraint: a probabilistic analysis}
\maketitle

\begin{abstract}
We consider a constrained version of the shortest path problem on the complete graphs whose edges have independent random lengths and costs. We establish the asymptotic value of the minimum length as a function of the cost-budget within a wide range.
\end{abstract}

\bigskip

\begin{footnotesize}
\noindent {\em 2020 Mathematics Subject Classification.} 05C80, 90C27, 60C05.

\noindent {\em Key words.} Random shortest path, cost constraint, weighted graph.
\end{footnotesize}

\bigskip

\section{Introduction}

Let the edges of the complete graph $K_n$ be given independent random edge lengths $w(e)$ and random costs $c(e)$ for $e\in E_n=E(K_n)$. Suppose further that we are given a budget $c_0$ and we need to find a path $P$ from vertex 1 to vertex 2 of minimum length $w(P)=\sum_{e\in P}w(e)$ whose cost $c(P)=\sum_{e\in P}c(e)$ satisfies $c(P)\leq c_0$. More precisely, let $\cP$ denote the set of paths from 1 to 2 in $K_n$. We wish to solve
\[
\text{{\bf CSP:} minimize }w(P)\text{ subject to }P\in \cP,c(P)\leq c_0. 
\] 
This is a well studied problem, at least in the worst-case, see for example Chen and Nie \cite{CN}, Climaco and Martins \cite{CM}, Machuca, Mandow, P\'erez de la Cruz and Ruiz-Sepulveda \cite{MMPR}, Nielsen, Pretolani and Anderesen \cite{NPA}, Pascoal, Captivo and Cl\'imaco \cite{PCC}.

In this paper we consider the case where $w(e),c(e),e\in E_n$ are independent random variables and we let $L_n=L_n(c_0)$ denote the random minimum length of a within budget shortest path. Also, let $H_n$ denote the {\em hop-count} (number of edges) in the shortest such path.  In particular we will assume that $w(e),c(e)$ are independent copies of the uniform $[0,1]$ random variable $U$. In a recent paper Frieze, Pegden, Sorkin and Tkocz \cite{FPT} considered a slightly more general setting, but were only able to bound $L_n$ w.h.p. between two values. In the simpler setting of this paper we are able to get an asymptotically correct estimate of $L_n$. 

{\bf Notation:} we say that $A_n\lesssim B_n$ if $A_n\leq (1+o(1))B_n$ as $n\to \infty$; $A_n\gtrsim B_n$ if $A_n\geq (1+o(1))B_n$ as $n\to \infty$; $A_n\approx B_n$ if $A_n\lesssim B_n$ and $A_n\gtrsim B_n$.

We prove the following: 
\begin{theorem}\label{th1}
Suppose that $w(e),c(e),e\in E_n$ are independent copies of a uniform random variable on $[0,1]$. Suppose that 
\begin{equation}\label{eq:c0-assum}
\frac{1}{\sqrt{2}}\frac{\log^2n}{n}\leq c_0 \leq \frac{1}{2\sqrt{2}}. 
\end{equation}
Then w.h.p.
\[
L_n\approx \frac{\log^2n}{4c_0n}\qquad\text{ and }\qquad H_n\approx \frac{\log n}{2}.
\]
\end{theorem}
The main new ideas of the paper are in the proof of Theorem \ref{th1}. With a little effort this theorem can be generalised to prove the following:
\begin{theorem}\label{th2}
Suppose now that $w(e),c(e),e\in E_n$ are independent copies of $U^\g$ where $0<\g\leq 1$. Suppose that 
\[
\frac{a_1\log^2n}{n}\leq c_0\leq a_2\qquad\text{ for constants $a_1,a_2$ dependent on $\g$.} 
\]
Then w.h.p.
\[
L_n\approx  \frac{\g\log^2n}{4\G\brac{\frac{1}{\g}+1}^2c_0n^\g}\qquad\text{ and }\qquad H_n\approx \frac{\g\log n}{2}.
\]
\end{theorem}

\section{Outline of paper}

We will obtain an estimate of $L_n$ in two distinct ways and combine them to give us what we need. In Section \ref{s1} we use the first moment method to get a lower bound and in Section \ref{s2} we use Lagrangean Duality to obtain another bound. We combine the two bounds and finish the proof of Theorem \ref{th1} in Section \ref{s3}. We then give a sketch proof of Theorem \ref{th2} in Section \ref{Mogen}.

\section{First Moment}\label{s1}

The goal of this section is to prove the following lemma about a high probability bound on the product $w(P)c(P)$ for \emph{every} path $P \in \cP$.

\begin{lemma}\label{lm:wc}
Suppose that $w(e),c(e),e\in E_n$ are independent copies of $U$. W.h.p. for every path $P \in \cP$, we have
\[
w(P)c(P) \gtrsim \frac{\log^2n}{4n}.
\]
\end{lemma}

\begin{corollary}\label{cor:Ln-low-bd}
For every $c_o$, $L_n \gtrsim \frac{\log^2n}{4nc_0}$ w.h.p.
\end{corollary}
\begin{proof}
We have, $L_n = \min\{w(P): \ P \in \cP, c(P) \leq c_0\} \geq \min\left\{w(P)\frac{c(P)}{c_0}: \ P \in \cP, c(P) \leq c_0\right\} \gtrsim \frac{\log^2n}{4nc_0}$.
\end{proof}

For the proof of Lemma \ref{lm:wc}, we need a bound on events that for a fixed path $P$, we have $w(P)c(P) \leq t$.

\begin{lemma}\label{lm:ST}
Let $S$ and $T$ be independent copies of $U_1 + \dots + U_k$, where $U_1, \ldots, U_k$ are i.i.d. copies of a uniform random variable on $[0,1]$. Then for $0 < t < k^2$, we have
\[
\Pr\brac{ST \leq t} \leq \frac{t^k}{k!^2}\left(k\log\left(\frac{k^2}{t}\right) + 2\frac{k!}{k^k}\right).
\]
\end{lemma}
\begin{proof}
We have $\Pr\brac{S \leq x} \leq \frac{x^k}{k!}$ for every $x \geq 0$ (e.g. by looking at the volume of the orthogonal simplex of side-length $x$). Moreover, $S \leq k$, so $\Pr\brac{S \leq x} = 1$ for every $x \geq k$. Therefore, using independence,
\begin{align}
\Pr\brac{ST \leq t} = \E_T\Pr_S\brac{S \leq \frac{t}{T}} &= \E_T\brac{\Pr_S\brac{S \leq \frac{t}{T}}\ind{\frac{t}{T} < k}} +  \E_T\ind{\frac{t}{T} \geq k} \nonumber\\
&\leq \E_T\brac{\frac{1}{k!}\left(\frac{t}{T}\right)^k\ind{T >\frac{ t}{k}}} + \Pr_T\brac{T \leq \frac{t}{k}}.\label{eq1}
\end{align}
For the second term, we use again $\Pr_T\brac{T \leq \frac{t}{k}} \leq \frac{1}{k!}\left(\frac{t}{k}\right)^k$, whereas for the first one, writing 
\[
T^{-k} = \int_0^\infty ku^{-k-1}\ind{u > T}\dd u\]
yields
\begin{align*}
\E_T\brac{\frac{1}{k!}\left(\frac{t}{T}\right)^k\ind{T > \frac{t}{k}}} &= \frac{t^k}{k!}\E_T\brac{\int_{0}^\infty ku^{-k-1}\ind{T > \frac{t}k, T < u} \dd u} \\
&\leq \frac{t^k}{k!}\int_{t/k}^\infty ku^{-k-1}\Pr\brac{T < u} \dd u \\
&\leq \frac{t^k}{k!}\left(\int_{t/k}^k ku^{-k-1}\frac{u^k}{k!} \dd u + \int_{k}^\infty ku^{-k-1} \dd u\right) \\
&= \frac{t^k}{k!^2}\left(k\log\left(\frac{k^2}{t}\right) + \frac{k!}{k^k}\right).
\end{align*}
Putting these together finishes the proof.
\end{proof}

\begin{proof}[Proof of Lemma \ref{lm:wc}]
For constant $\b>0$ to be chosen soon, we let
\[
\cP_\b=\set{P\in \cP:w(P)c(P)\leq \frac{\b\log^2n}{n}}.
\]
Then,
\beq{0}{
\E(|\cP_\b|)\leq\sum_{1 \leq \ell \leq n}n^{\ell-1}\frac{1}{\ell!^2}\bfrac{\b\log^2n}{n}^\ell \left(\ell\log\left(\frac{\ell^2n}{\b\log^2n}\right) + 2\frac{\ell!}{\ell^\ell}\right).
}
Explanation: we choose the $\ell-1$ internal vertices and order them in $\binom{n}{\ell-1}(\ell-1)!\leq n^{\ell-1}$ ways to create a path $P$ of edge-length $\ell$. We then use Lemma \ref{lm:ST} to bound the probability that $w(P)c(P) \leq \frac{\b\log^2n}{n}$, i.e. $P \in \cP_\beta$.

Let $u_\ell$ denote the summand in \eqref{0}. Note that for large enough $n$,
\[
\ell\log\left(\frac{\ell^2n}{\b\log^2n}\right) + 2\frac{\ell!}{\ell^\ell} \leq \ell\log\left(\frac{n^3}{\b\log^2n}\right) + 2 \leq 3\ell\log n.
\]
Using $\ell! \geq (\ell/e)^\ell$, putting $\ell = \alpha \log n$ and $\D=\a(\log(\b/\a^2)+2)-1$, we have
\[
n^{\ell-1}\frac{1}{\ell!^2}\bfrac{\b\log^2n}{n}^\ell = e^{\D\log n}.
\]
Looking at $\frac{\partial \D}{\partial\a} = \log \frac{\beta}{\a^2}$, we see that for a fixed $\beta$, $\D$ is maximized when $\alpha = \sqrt{\beta}$, giving $\D \leq 2\sqrt{\beta} - 1$. Therefore, we choose $\beta$ such that, say $2\sqrt{\beta} - 1 = -(\log n)^{-1/2}$, that is
\[
\beta = \frac{1}{4}\left(1 - \frac{1}{\sqrt{\log n}}\right)^2.
\]
Then
\[
\sum_{1 \leq \ell \leq \log^2n} u_\ell \leq \sum_{1 \leq \ell \leq \log^2n} e^{-\sqrt{\log n}}\cdot 3\ell \log n \leq 3\log^5ne^{-\sqrt{\log n}} = o(1).
\]
It remains to note that for $\ell > \log^2n$,
\[
n^{\ell-1}\frac{1}{\ell!^2}\bfrac{\b\log^2n}{n}^\ell \leq n^{-1} \left(\frac{e^2\b\log^2n}{\ell^2}\right)^\ell \leq n^{-1}\left(\frac{e^2\b\log^2n}{\log^4n}\right)^\ell \leq n^{-1}e^{-\ell} \leq n^{-1}e^{-\log^2n},
\]
thus
\[
\sum_{\log^2n < \ell \leq n} u_\ell \leq \sum_{\log^2n < \ell \leq n} n^{-1}e^{-\log^2n}\cdot 3\ell\log n \leq 3n(\log n) e^{-\log^2n} = o(1).
\]
We conclude that $\Pr\brac{|\cP_\b|>0} \leq\E|\cP_\b|=o(1)$ with $\beta = \frac{1}{4}\left(1 - \frac{1}{\sqrt{\log n}}\right)^2$.
\end{proof}

\begin{Remark}\label{rem:easier}
Corollary \ref{cor:Ln-low-bd} can also be obtained by a much simpler first moment argument counting the number of paths $P$ such that $w(P) \leq w$ and $c(P) \leq c$ with $wc \gtrsim \frac{\log^2n}{4n}$ (see \cite{FPT}). Lemma \ref{lm:wc} will be crucial however in Section \ref{s3}.
\end{Remark}

\section{The dual}\label{s2}

In this section, motivated by the Lagrangean dual to CSP, we consider the following random variable
\[
\psi(\l)=\min\set{w(P)+\l c(P):\;P\in \cP},
\]
where $\lambda$ is a positive parameter chosen later.
In words, $\psi(\lambda)$ is the the minimum length of a path from 1 to 2 in $K_n$ when edge lengths are independent copies of $W=W(\l)=U_1+\l U_2$. Here $U_1,U_2$ are independent copies of $U$. The goal of this section is to establish the following upper bound on $\psi(\lambda)$ for a specific \emph{optimal} choice of $\lambda$ (see Remark \ref{rem:l*} below).

\begin{lemma}\label{lm:psi-upp-bd}
Let $c_0$ satisfy \eqref{eq:c0-assum}. Let $\lambda^* = \frac{\log^2n}{4c_0^2n}$. Then
\begin{equation}\label{eq:psi-upp-bd}
\psi(\lambda^*) \lesssim \frac{\log^2n}{2c_0n} \qquad \text{w.h.p.}
\end{equation}
\end{lemma}

We shall need the following result of Bahmidi and van der Hofstadt from \cite{BH}.
\begin{theorem}[\cite{BH}]\label{thm:BH}
Let $s \in (0,1)$ be a constant. Let $L_{s,n}$ be the length of a shortest path from $1$ to $2$ in the complete graph $K_n$ when edge lengths are independent copies of $\xi^s$, where $\xi$ is an exponential mean $1$ random variable. We have the following convergence in distribution
\beq{BaHo}{
n^s L_{s,n}-\frac{1}{\G(1+1/s)^s}\log n\stackrel{d}{\longrightarrow} Z,
}
for some random variable $Z$. Moreover, for the hop-count $H_{s,n}$,
\[
H_{s,n}\approx s\log n\qquad\text{w.h.p.}
\]
\end{theorem}

In fact, we shall only need the following simple consequence of \eqref{BaHo}.

\textbf{Claim.} $L_{s,n} \approx \frac{1}{\G(1+1/s)^s}\frac{\log n}{n^s}$ w.h.p.

\begin{proof}
For every sequence of numbers $a_n \to 0$, from \eqref{BaHo}, we get
\[
a_n\left(n^s L_{s,n}-\frac{1}{\G(1+1/s)^s}\log n\right) \to 0 \qquad \text{in probability}.
\]
Choosing, say $a_n = (\log n)^{-1/2}$, we get
\[
\Pr\brac{(\log n)^{-1/2}\left|n^s L_{s,n}-\frac{1}{\G(1+1/s)^s}\log n\right| > 1} \to 0,
\]
or, in other words,
\[
L_{s,n} = \frac{\log n}{\G(1+1/s)^sn^s}\Big[1+\theta_n(\log n)^{-1/2}\G(1+1/s)^s\Big]
\]
with $\Pr(|\theta_n| < 1) \to 1$.
\end{proof}

Heuristically, the idea is that the density of $W(\lambda)$ near the origin behaves like the density of $(2\l\xi)^{1/2}$, hence $\psi(\lambda)$ is asymptotic to $(2\l)^{1/2}L_{1/2,n}$ whose asymptotic behaviour is in turn governed by \eqref{BaHo}. 
To make this rigorous, we need the following lemma.

\begin{lemma}\label{lm:trun}
Let $L_{1/2,n}$, $\tilde L_{1/2,n}$ be the length of a shortest path from $1$ to $2$ in the complete graph $K_n$ when edge lengths are independent copies of $\xi^{1/2}$, $\tilde \xi^{1/2}$ respectively, where $\xi$ is an exponential mean $1$ random variable and
\[
\tilde \xi = \begin{cases} \xi, & \xi \leq \frac{\log^2 n}{n}, \\
\infty, & \xi > \frac{\log^2 n}{n}.  \end{cases}
\]
Then $L_{1/2,n} = \tilde L_{1/2,n}$ w.h.p.
\end{lemma}
\begin{proof}
Now the claim applied with $s = \frac{1}{2}$ implies that w.h.p. the shortest path from 1 to 2 has length $\lesssim \frac{1}{\sqrt{2}}\frac{\log n}{\sqrt{n}}$ if edge lengths are given by $\xi^{1/2}$. This clearly implies that w.h.p. {\em the} shortest path $P$ from 1 to 2 contains no edge with $\xi^{1/2} > \frac{\log n}{\sqrt{n}}$. Replacing $\xi$ by $\tilde \xi $ can only increase path lengths and by the previous sentence, $P$ will w.h.p. still have the same length. This implies the lemma.
\end{proof}

We proceed with the proof of Lemma \ref{lm:psi-upp-bd}.

\begin{proof}[Proof of Lemma \ref{lm:psi-upp-bd}]
We split the argument into two cases depending on the value of $\l^*$.

\textbf{Case 1.} $\l^* \geq 1$. For $t \leq 1$, we have
\begin{equation}\label{eq:tails}
\Pr(W\leq t) = \Pr(U_1 + \l^*U_2 \leq t) = \frac{t^2}{2\l^*} \geq 1 - e^{-\frac{t^2}{2\l^*}} = \Pr\brac{(2\l^*\xi)^{1/2}\leq t}.
\end{equation}
Observe that for $\tilde \xi$ from Lemma \ref{lm:trun}, we have 
\[
\Pr\brac{\tilde \xi \leq t} = \begin{cases} \Pr\brac{\xi \leq t}, & t \leq \frac{\log^2n}{n}, \\ \Pr\brac{\xi \leq \frac{\log^2n}{n}}, & t > \frac{\log^2n}{n},  \end{cases}
\]
for every $t \geq 0$. Therefore, the following comparison holds
\begin{equation}\label{eq:tails'}
\Pr(W\leq t)  \geq \Pr\brac{(2\l^*\tilde \xi)^{1/2}\leq t},
\end{equation}
for every $t \geq 0$ as long as $\frac{1}{2\l^*} \geq \frac{\log^2n}{n}$, equivalently $c_0 \geq \frac{\log^2n}{2^{1/2}n}$, which is assumed in \eqref{eq:c0-assum}.
This means that $W$ is stochastically dominated by $(2\l^*\tilde \xi)^{1/2}$. As a result,
\[
\psi(\l^*) \leq \sqrt{2\l^*}\tilde L_{1/2,n} = \sqrt{\frac{\log^2n}{2c_0^2n}}L_{1/2,n},
\]
where the equality follows from Lemma \ref{lm:trun}. The claim made after Theorem \ref{thm:BH} gives $L_{1/2,n} \approx \frac{\log n}{\sqrt{2n}}$, which finishes the argument.

\textbf{Case 2.} $\l^* \leq 1$. We repeat the whole argument of Case 1. The only change is that now \eqref{eq:tails} holds for all $t \leq \l^*$ instead of all $t \leq 1$, thus to establish \eqref{eq:tails'} for all $t \geq 0$, we need $\frac{\l^*}{2} \geq \frac{\log^2n}{n}$, equivalently, $c_0 \leq \frac{1}{2\sqrt{2}}$.
\end{proof}

\begin{Remark}\label{rem:c0}
An inspection of the proof shows that the implicit $o(1)$ term in \eqref{eq:psi-upp-bd} does not depend on $c_0$. 
\end{Remark}

\begin{Remark}\label{rem:l*}
The value $\lambda^*$ was chosen so as to minimize $\psi(\lambda) - \lambda c_0$ (the dual lower bound on CSP), where we put $\psi(\lambda) = \sqrt{\lambda}\frac{\log n}{\sqrt{n}}$ (heuristically $\psi(\lambda) \approx \sqrt{2\l}L_{1/2,n} \approx \sqrt{\lambda}\frac{\log n}{\sqrt{n}}$).
\end{Remark}

\section{Proof of Theorem \ref{th1}}\label{s3}

In view of Corollary \ref{cor:Ln-low-bd}, we need to upper bound $L_n$, or in other words, show that w.h.p. there is a path $P$ with $w(P) \leq \frac{log^2n}{4nc_0}$ and $c(P) \leq c_0$.

Let $\l^* = \frac{\log^2n}{4c_0^2n}$. By the definition of $\psi(\l^*)$, we get a path $P$ of length $w=w(P)$ and cost $c=c(P)$ that w.h.p. satisfies
\begin{align}
wc&\gtrsim\frac{\log^2n}{4n}\qquad\text{ from Lemma \ref{lm:wc},}\label{f1}\\
w+\frac{\log^2n}{4c_0^2n}c&\lesssim \frac{\log^2n}{2c_0n}\qquad\text{ from Lemma \ref{lm:psi-upp-bd}}.\label{f2}
\end{align}
The implicit $o(1)$ terms here do not depend on $c_0$ (which is clear for \eqref{f1} and is justified by Remark \ref{rem:c0} for \eqref{f2}). 
Combining \eqref{f1} and \eqref{f2} yields
\begin{align*}
(1-o(1))\frac{\log^2n}{4n} \leq wc \leq \left[\frac{\log^2n}{2c_0n}(1+o(1)) - \frac{\log^2n}{4c_0^2n}c\right]c = \frac{\log^2n}{4nc_0}\left(2(1+o(1))-\frac{c}{c_0}\right)c,
\end{align*}
thus, in terms of $r = \frac{c}{c_0}$,
\[
1-o(1) \leq 2(1+o(1))r - r^2\quad\text{ or }\quad1-\e\leq 2(1+\e)r-r^2
\]
for some $\e=\e(n)\to 0$. Re-arranging gives 
\[
(r-1-\e)^2\leq 3\e+\e^2\text{ and so }c\leq c_0(1+2\e^{1/2}). 
\]
Note now that \eqref{f2} implies that
\beq{f3}{
w\leq \frac{\log^2n}{4c_0n}\brac{2+o(1)-\frac{c}{c_0}}\approx  \frac{\log^2n}{4c_0n}.
}
Now let $\hc_0=c_0(1-2\e^{1/2})$ and repeat the above analysis with $\hc_0$ replacing $c_0$. Then w.h.p. we see that w.h.p. there is a path of length at most $\frac{\log^2n}{4\hc_0n}\approx \frac{\log^2n}{4c_0n}$ and cost at most $\hc_0(1+2\e^{1/2})\leq c_0$. This completes the proof of Theorem \ref{th1}.

\section{More general distributions}\label{Mogen}

The goal is to sketch a proof of Theorem \ref{th2}. We first have to generalise Lemma \ref{lm:ST}. For this we need the following lemma.
\begin{lemma}\label{lema}
Let $\g > 0$. Let $U_1,U_2,\ldots,U_k$ be independent copies of a uniform random variable on $[0,1]$. Then, for $u\geq 0$, we have
\[
\Pr(U_1^\g+U_2^\g+\dots+U_k^\g\leq u)\leq\frac{u^{k/\g}\G\brac{\frac{1}{\g}+1}^k}{\G\brac{\frac{k}{\g}+1}}.
\]
\end{lemma}
\begin{proof}
We have
\begin{align*}
\Pr(U_1^\g+U_2^\g+\dots+U_k^\g \leq u) &=\text{Vol}\set{x\in [0,1]^k:\;\sum_{i=1}^k x_i^\g \leq u} \\
&\leq \text{Vol}\set{x\in [0,\infty)^k:\;\sum_{i=1}^k x_i^\g \leq u} = u^{k/\g}v_{k,\g}
\end{align*}
where $v_{k,\g}=\text{Vol}\set{x\in [0,\infty)^k:\;\sum_{i=1}^k x_i^\g \leq 1}$. A standard computation leads to a closed expression, 
\mults{
\left(\int_0^\infty e^{-t^\g} dt\right)^k = \int_0^\infty\dots\int_0^\infty e^{-x_1^\g-\dots-x_k^\g} dx_1\dots dx_k = \int_0^\infty\dots\int_0^\infty \int_{s > x_1^\g+\dots+x_k^\g} e^{-s} ds dx_1\dots dx_k \\
= \int_{s>0} e^{-s}\text{Vol}\set{x\in [0,\infty)^k:\;\sum_{i=1}^k x_i^\g<s} ds 
= \int_{s>0} e^{-s}s^{k/\g}v_{k,\g} ds, 
}
and thus
\[
v_{k,\g} = \frac{\left(\int_0^\infty e^{-t^\g} dt\right)^k}{\int_{s>0} e^{-s}s^{k/\g}ds} = \frac{\Gamma\brac{\frac{1}{\g}+1}^k}{\Gamma(\frac{k}{\g}+1)}.
\]
\end{proof}
Given this we have
\begin{lemma}\label{lm:STa}
Let $S$ and $T$ be independent copies of $U_1^\g + \dots + U_k^\g$, where $U_1, \ldots, U_k$ are i.i.d. copies of a uniform random variable on $[0,1]$. Then for $0 < t < k^2$, we have
\[
\Pr\brac{ST \leq t} \leq \frac{t^{k/\g}\G\brac{\frac{1}\g+1}^{2k}} {\G\brac{\frac{k}\g+1}^2}\brac{\frac{k}{\g}\log\bfrac{k^2}{t}+ 2\frac{\G\brac{\frac{k}\g+1}}{k^k\G\brac{\frac{1}\g+1}^k}}.
\]
\end{lemma}
\begin{proof}
We repeat the proof of Lemma \ref{lm:ST}. The bound in \eqref{eq1} becomes 
\begin{align*}
&\E_T\brac{\frac{\bfrac{t}{T}^{k/\g}\G\brac{\frac{1}\g+1}^k}{\G\brac{\frac{k}\g+1}}1\brac{T>\frac{t}k}} +\frac{\bfrac{t}{k}^{k/\g}\G\brac{\frac{1}\g+1}^k}{\G\brac{\frac{k}\g+1}}\\
&\leq
\frac{t^{k/\g}\G\brac{\frac{1}\g+1}^k}{\G\brac{\frac{k}\g+1}}\brac{\frac{k}{\g}\frac{\G\brac{\frac{1}\g+1}^k}{\G\brac{\frac{k}\g+1}}\log\bfrac{k^2}{t}+\frac{1}{k^{k/\g}}} +\frac{\bfrac{t}{k}^{k/\g}\G\brac{\frac{1}\g+1}^k}{\G\brac{\frac{k}\g+1}}\\
&=\frac{t^{k/\g}\G\brac{\frac{1}\g+1}^{2k}} {\G\brac{\frac{k}\g+1}^2}\brac{\frac{k}{\g}\log\bfrac{k^2}{t}+ 2\frac{\G\brac{\frac{k}\g+1}}{k^{k\g}\G\brac{\frac{1}\g+1}^k}}.
\end{align*}
\end{proof}

\begin{proof}[Proof of Theorem \ref{th2} (Sketch).]
We define
\[
\cP_\b=\set{P\in \cP:w(P)c(P)\leq \frac{\b\log^2n}{n^\g}}.
\]
Using Lemma \ref{lm:STa}, \eqref{0} becomes, 
\[
\E(|P_\b|))\leq \sum_{\ell\geq 1}n^{\ell-1} \frac{(\b\log^2n)^{\ell/\g}\G\brac{\frac{1}{\g}+1}^{2\ell}}{n^\ell\G\brac{\frac{\ell}{\g}+1}^2}\brac{\frac{\ell}{\g}\log\bfrac{\ell^2}{t}+ 2\frac{\G\brac{\frac{\ell}\g+1}}{\ell^\ell\G\brac{\frac{1}\g+1}^\ell}} .
\]
We deduce from this that w.h.p.
\beq{x1}{
w(P)c(P)\gtrsim \frac{\log^2n}{4\G\brac{\frac{1}\g+1}^{2\g} n^\g}.
}
To consider the dual problem we use that if $U_1,U_2$ are independent copies of $U$, then
\[
\Pr(U_1^\g+\l U_2^\g\leq t)=\frac{t^{2/\g}\G\brac{\frac{1}{\g}+1}^2}{\l^{1/\g}\G\brac{\frac{2}{\g}+1}}
\]
valid for $0<t<1\leq \l$, see equation (36) of \cite{FT}.

As in Section \ref{s2}, thanks to \eqref{BaHo} (with $s = \g/2$) and stochastic dominance (an analogue of Lemma \ref{lm:trun}), we obtain
\beq{x2}{
\psi(\l)\lesssim \frac{\sqrt{\l^*}\log n}{\G(1+\frac1\g)^\g n^{\g/2}} = \frac{\log^2n}{2c_0\G(1+\frac1\g)^{2\g}n^\g}
}
with $\l^* = \frac{1}{4c_0^2\G(1+\frac1\g)^{2\g}}\frac{\log^2n}{n^\g}$ (chosen to minimise $\frac{\sqrt{\l}\log n}{\G(1+\frac1\g)^\g n^{\g/2}} - \l c_0$).
Applying the analogous argument in Section \ref{s3} to \eqref{x1}, \eqref{x2} we see finally that w.h.p.
\[
L_n\approx \frac{\g\log^2n}{4\G\brac{\frac{1}{\g}+1}^{2\g}c_0n^\g}\quad\text{ and }\quad H_n\approx \frac{\g\log n}{2}.
\]
A coupling argument of Janson \cite{Jan} can be used for the case where $w(e),c(e)$ have the distribution function $F_w(t)=\Pr(X\leq t)$, of a random variable $X$, that satisfies $F(t)\approx at^{1/\g},\g\leq 1$ as $t\to 0$. This argument is spelled out in detail in Section 4.1 of \cite{FT}.
\end{proof}

\section{Final Remarks}
We first observe that our proof shows that w.h.p. the duality gap between the maximum dual value and the optimal value of the solution to the constrained shortest path problem is within $o(1)$ of the optimal value to the latter problem.

The imposed range \eqref{eq:c0-assum} on $c_0$ is a by-product of our proof. It is likely off by a factor $\log n$ on both sides: a lower bound on $c_0$ of $\approx \frac{\log n}{n}$ comes from the unconstrained minimum cost of a path, whereas if $c_0\geq \log n$, then w.h.p. the unconstrained minimum length path will be within cost budget.

We would next like to mention the fact that the approach of Beier and Voeking \cite{BV1} can be applied to solve the computational problem in polynomial expected time. This paper was the first (and only?) paper to give a polynomial expected time algorithm for solving random 0-1 knapsack problems. In Theorem 2 of this paper, they give a significant generalisation which opens the door for solving the constrained shortest path problem.

Theorem 2 of that paper is
\begin{theorem}
Let $S_1,S_2,\ldots,S_m$ be a fixed but arbitrary sequence of subsets of $[N]$. Suppose that profits are chosen according to the uniform distribution over $[0,1]$. Let $q$ denote the number of dominating sets over $S_1,S_2,\ldots,S_m$. Then $\E(q)=O(N^3)$. 
\end{theorem}
To unpack this, we first observe that $q$ determines the running time of an algorithm of Nemhauser and Ullman \cite{NU} that can be used to solve the knapsack problem. Here $q$ is the number of sets among $S_1,S_2,\ldots,S_m$ that are not dominated by any other set. Here $S_i$ dominates $S_j$ if it has smaller cost and larger profit.  To apply the theorem we let $N=\binom{n}{2}$ and $[N]$ be associated with the edge set of $K_n$. Then we let the profit $p(e)$ of edge $e$ be equal to $1-w(e)$. We then apply the theorem separately for each $\ell=1,2,\ldots,n-1$ and let the $S_i$ correspond to the set of edges in the $\ell$-edge paths from 1 to 2. In this way we can solve the constrained shortest path problem in $O(n^7)$ expected time.

The paper \cite{FPT} allowed multiple constraints and it is a challenge to tighten the result there to get an asymptotic result, as we did here.

\end{document}